\documentclass[a4paper, 10pt, oneside, onecolumn]{article}

\RequirePackage{geometry}
\usepackage[utf8]{inputenc}
\usepackage[T1]{fontenc}
\usepackage{lmodern}
\usepackage{amsmath}
\usepackage{amssymb}
\usepackage{amsfonts}
\usepackage{amsthm}
\usepackage{xcolor}
\usepackage{hyperref}
\usepackage{cleveref}
\usepackage{mathtools}
\usepackage{constants}
\usepackage[autostyle=true]{csquotes}

\newconstantfamily{C}{symbol=C}

\hypersetup{
	colorlinks=false,
	pdfborder={0 0 0},
	pdftitle={Can a chemotaxis-consumption system recover from a measure-type aggregation state in arbitrary dimension?},
	pdfauthor={Frederic Heihoff},
	pdfkeywords={},
	bookmarksopen=true,
}

\renewcommand{\phi}{\varphi}

\newtheorem{base}{Base}[section]
\numberwithin{equation}{section}

\theoremstyle{plain}
\newtheorem{theorem}[base]{Theorem}
\newtheorem{lemma}[base]{Lemma}

\newtheorem{corollary}[base]{Corollary}

\theoremstyle{definition}

\newcommand{\R}{\mathbb{R}}
\newcommand{\N}{\mathbb{N}}
\renewcommand{\d}{\,\mathrm{d}}
\newcommand{\laplace}{\Delta}
\newcommand{\grad}{\nabla}
\renewcommand{\div}{\nabla \cdot}
\renewcommand{\L}[1]{{L^{#1}(\Omega)}}

\newcommand{\defs}{\coloneqq}
\newcommand{\sfed}{\eqqcolon}
\newcommand{\stext}[1]{\;\;\text{ #1 }\;\;}
\newcommand{\eps}{\varepsilon}

\newcommand{\tmaxeps}{T_{\mathrm{max}, \eps}}

\newcommand{\ue}{u_\eps}
\newcommand{\ve}{v_\eps}

\newcommand{\uet}{u_{\eps t}}
\newcommand{\vet}{v_{\eps t}}

\newcommand{\Mp}{\mathcal{M}_+(\overline{\Omega})}

\newif\ifclarification
\clarificationtrue

\newcommand\numberthis{\addtocounter{equation}{1}\tag{\theequation}}

\makeatletter
\g@addto@macro\bfseries{\boldmath}
\makeatother

\title{Can a chemotaxis-consumption system recover from a measure-type aggregation state in arbitrary dimension?}
\author{
	Frederic Heihoff\footnote{fheihoff@math.uni-paderborn.de}\\
	{\small Institut f\"ur Mathematik, Universit\"at Paderborn,}\\
	{\small 33098 Paderborn, Germany}
}
\date{}
\begin{document}

\maketitle

\begin{abstract}
\noindent
We consider the chemotaxis-consumption system 
\[
	\left\{
		\begin{aligned}
			u_t &= \Delta u - \chi \nabla \cdot (u\nabla v) \\
			v_t &= \Delta v - uv 
		\end{aligned}
		\right.
\]
in a smooth bounded domain $\Omega \subseteq \mathbb{R}^n$, $n \geq 2$, with parameter $\chi > 0$ and Neumann boundary conditions. It is well known that, for sufficiently smooth nonnegative initial data and under a smallness condition for the initial state of $v$, solutions of the above system never blow up and are even globally bounded. Going in a sense a step further in this paper, we ask the question whether the system can even recover from an initial state that already resembles measure-type blowup. To answer this, we show that, given an arbitrarily large positive Radon measure $u_0$ with $u_0(\overline{\Omega}) > 0$ as the initial data for the first equation and a nonnegative $L^\infty(\Omega)$ function $v_0$ with 
\[
	0 < \|v_0\|_{L^{\infty}(\Omega)} < \frac{2}{3n\chi}
\] as initial data for the second equation, it is still possible to construct a global classic solution to the above system.\\[0.5em]
\end{abstract}

\section{Introduction}
Inspired by the seminal paper due to Keller and Segel from 1970 (cf.\ \cite{KellerInitiationSlimeMold1970}) studying the movement mechanism of \emph{dictyostelium discoideum} slime mold, recent decades have seen increasing exploration of similar partial differential equation models for the movement behavior of various biological organisms. Arguably the lion's share of models found in this lineage are still interested in the same chemotactic movement mechanism as the original case study, which is the directed movement of biological entities toward a chemical stimulus, or one of its close relatives (cf.\ \cite{BellomoMathematicalTheoryKeller2015}). A central question in the mathematical investigation of such models is whether global classical solutions always exist or whether there are initial states leading to finite-time blowup due the solution e.g.\ aggregating to a Dirac measure. Notably, the latter can be a desired property of model solutions if similar aggregation behavior was observed in experiments as is the case for the original Keller--Segel model. In this classic setting, it in fact depends on the dimension of the domain, the system parameters as well as some properties of the initial data whether global solutions always exist (cf.\ e.g.\ \cite{NagaiApplicationTrudingerMoserInequality1997}, \cite{WinklerAggregationVsGlobal2010}) or finite-time blowup occurs (cf.\ e.g.\ \cite{NagaiBlowupRadiallySymmetric1995}, \cite{WinklerFinitetimeBlowupHigherdimensional2013}).
\\[0.5em]
Building on the aforementioned blowup results for the Keller--Segel model, it has further been shown that in some cases the blowup state persists in the form of Dirac-type singularities even after blowup time (cf.\ \cite{BilerRadiallySymmetricSolutions2008}, \cite{LankeitFacingLowRegularity2020}). This suggests that recovery from blowup is not possible in these scenarios and naturally leads us to the following question: Do there exist sufficiently favorable scenarios in which model solutions can always recover from a blowup state? At least in some easier to handle versions of the Keller--Segel system, such a recovery from blowup has in fact been shown to be possible (\cite{HeihoffDoesStrongRepulsion2022}, \cite{WinklerInstantaneousRegularizationDistributions2019}). 
\\[0.5em]
Given this precedent, we now turn our attention to a closely related setting to investigate the same question for the chemotaxis-consumption model
\begin{equation}\label{problem}
	\left\{
	\begin{aligned}
		u_t &= \laplace u - \chi \div (u\grad v) && \text{ on } \Omega \times (0,\infty) \\
		v_t &= \laplace v - uv && \text{ on } \Omega \times (0,\infty) \\
		0 &= \grad u \cdot \nu = \grad v \cdot \nu && \text{ on } \partial \Omega \times (0,\infty)
	\end{aligned}
	\right.
\end{equation}
in a smooth bounded domain $\Omega \subseteq \R^n$, $n \geq 2$, with $\chi > 0$. It differs from the classic Keller--Segel system only due to a change in the second equation, namely the production terms have been replaced by the consumption term $-uv$. 
While often found coupled to an additional fluid equation, systems of this type are e.g.\ used to model certain bacteria that are attracted by oxygen (cf.\ \cite{TuvalBacterialSwimmingOxygen2005}).
\paragraph{Prior work.} We will now give a brief overview of some prior work regarding both general results about the system (\ref{problem}) as well as results regarding recovery from blowup in other contexts. For a broader overlook regarding chemotaxis results, we refer the reader to the surveys \cite{BellomoMathematicalTheoryKeller2015} and \cite{LankeitDepletingSignalAnalysis2023}.
\\[0.5em]
Under a smallness assumption for the initial state of the second solution component, it has been shown that bounded global classical solutions to (\ref{problem}) with smooth initial data exist in domains of arbitrary dimension (cf.\ 
\cite{BaghaeiBoundednessClassicalSolutions2017}, \cite{TaoBoundednessChemotaxisModel2011}). Removing the initial data condition, a classical existence result of this type still holds in two-dimensional domains (cf.\ \cite{JiangGlobalExistenceAsymptotic2015}) and it has further been shown that weak solutions exist in three-dimensional domains, which eventually become smooth (cf.\ \cite{taoEventualSmoothnessStabilization2012}). There has also been extensive exploration of chemotaxis-consumption models with modified taxis or diffusion mechanisms (cf.\ e.g.\ \cite{CaoGlobalintimeBoundedWeak2014a}, \cite{LiGlobalSmalldataSolutions2015}, \cite{LiBoundednessStabilizationChemotaxis2021}, \cite{ZhangBoundednessChemotaxisSystems2016}), added production terms (cf.\ e.g.\ \cite{LankeitGlobalExistenceBoundedness2017a}, \cite{LiBoundednessStabilizationChemotaxis2021}), as well as models coupled with a fluid equation (cf.\ e.g.\ \cite{AhnGlobalClassicalSolutions2021}, \cite{winklerGlobalWeakSolutions2016}, \cite{WinklerSmallmassSolutionsTwodimensional2020} for some relevant existence theory and \cite{WinklerSmallmassSolutionsTwodimensional2020}, \cite{ZhangConvergenceRatesSolutions2015} for a discussion of long time behavior in this setting).
\\[0.5em]
Regarding recovery from blowup in chemotaxis settings, the landscape of results is to our knowledge still rather sparse. For the classic Keller--Segel system, smooth solutions with irregular initial data have been constructed in one-dimensional domains in \cite{WinklerInstantaneousRegularizationDistributions2019}, in two-dimensional domains for the attractive-repulsive version of the system under a strong repulsion assumption in \cite{HeihoffDoesStrongRepulsion2022}, or under the regularizing influence of an additional logistic source term in \cite{LankeitIrregularInitialData}. There have also been some efforts to construct mild solutions in two-dimensional domains with a small measure as initial data (cf.\ \cite{BilerLocalGlobalSolvability1998}). Further, there has been some exploration of this question in the whole space case (cf.\ \cite{BedrossianExistenceUniquenessLipschitz2014}, \cite{BilerCauchyProblemSelfsimilar1995}, \cite{BilerLocalGlobalSolvability1998}, \cite{BilerExistenceSolutionsKellerSegel2015}, \cite{RaczynskiStabilityPropertyTwodimensional2009}), radially symmetric settings (cf.\ \cite{WangImmediateRegularizationMeasuretype2021}, \cite{WinklerHowStrongSingularities2019}) as well as toroidal settings (cf.\ \cite{SenbaWeakSolutionsParabolicelliptic2002}).

\paragraph{Main result.} The main result of this paper is that, even with highly irregular initial data, it is still possible to construct classical solutions to (\ref{problem}) assuming that the initial data for the second equation are sufficiently small. More precisely letting $\Mp$ be the set of positive Radon measures on $\overline{\Omega}$ with the vague topology, we prove the following theorem:
\begin{theorem}\label{theorem:main}
	Let $\Omega \subseteq \R^n$, $n \geq 2$, be a bounded domain with a smooth boundary and let $\chi > 0$.
	\\[0.5em]
	Then for all initial data $u_0 \in \Mp$ with $u_0(\overline{\Omega}) > 0$ and nonnegative initial data $v_0 \in L^\infty(\Omega)$ with 
	\begin{equation}\label{v_0_condition}
		0 < \|v_0\|_\L{\infty} < \frac{2}{3n\chi},	
	\end{equation}
	there exist nonnegative functions $u, v \in C^{2,1}(\overline{\Omega}\times(0,\infty))$ such that $(u,v)$ is a classical solution of (\ref{problem}) attaining the initial data $(u_0, v_0)$ in the following fashion:
	\begin{align}
		u(\cdot, t) &\rightarrow u_0 &&\text{ in } \Mp \label{u_continuity}\\
		v(\cdot, t) &\rightarrow v_0 &&\text{ in } L^p(\Omega) \text{ for all } p \in (1,\infty) \label{v_continuity} 
	\end{align}
	as $t \searrow 0$.
\end{theorem}\noindent
\paragraph{Approach.} As is not uncommon, we will construct our desired solutions as the limit of a family of approximate solutions $((\ue, \ve))_{\eps \in (0,1)}$. Our approach will be guided by both the classical existence result for small data solutions to (\ref{problem}) with smooth initial data by Tao (cf.\ \cite{TaoBoundednessChemotaxisModel2011}) as well as a more recent development regarding the construction of solutions with measure-valued initial data to the repulsive-attractive Keller--Segel system in \cite{HeihoffDoesStrongRepulsion2022}.
\\[0.5em]
To construct the aforementioned approximate solutions, we first smoothly approximate our initial data and then apply the ready-made local existence theory from \cite{TaoBoundednessChemotaxisModel2011} to the otherwise unchanged system (\ref{problem}). Ruling out finite-time blowup for these local solutions will be a convenient by-product of our other arguments. Using an adapted version of what is arguably the key innovation of \cite{TaoBoundednessChemotaxisModel2011}, we then derive the following uniform differential inequality in \Cref{lemma:magic_functional}: 
\begin{equation}\label{eq:explaination}
	\frac{1}{p}\frac{\d}{\d t} \int_\Omega \ue^p \phi(\ve) + \frac{2\delta(p-1)}{p^2} \int_\Omega \phi(\ve) |\grad \ue^\frac{p}{2}|^2 + \frac{\delta}{p}\int_\Omega \ue^p \phi''(\ve) |\grad \ve|^2 \leq 0
\end{equation}
with some small $\delta \in (0,1)$, $\phi(s) \defs e^{(\beta s)^2}$, $s \geq 0$, and $\beta \defs \sqrt{\frac{\chi(p-1)}{4\|v_0\|_\L{\infty}}}$ for all $p \in ( 1, \frac{n}{2} + \delta ]$. Given its origin, it may not be surprising that it is this argument that necessitates the introduction of the smallness condition (\ref{v_0_condition}). Notably while in \cite{TaoBoundednessChemotaxisModel2011} any a priori information thereby obtained for the term $\int_\Omega \ue^p \phi''(\ve) |\grad \ve|^2$ was discarded, it will play a crucial role in our later derivations.  
\\[0.5em]
In view of the potential irregularity of our initial data, we naturally cannot expect the functional $\int_\Omega \ue^p \phi(\ve)$ to be uniformly bounded at $t = 0$. Fortunately, the second term in (\ref{eq:explaination}), which is of dissipative type, suggests that our functional enjoys immediate boundedness properties independent of the initial state as long as there is at least a uniform mass bound available similar to e.g.\ the immediate smoothing properties of solutions to the heat equation. Instead of using a standard comparison argument to show that this is in fact the case, we will adapt a technique from \cite{HeihoffDoesStrongRepulsion2022}, which will allow us to further extract some additional a priori information for the third term in (\ref{eq:explaination}). More specifically by multiplying our functional by $t^\lambda$ with a sufficiently large $\lambda > 0$ and analyzing the time evolution of the new functional obtained in this way, we gain the following uniform bounds on any time interval $(0,T)$ in \Cref{lemma:time_scaled_functional}:
\begin{equation}\label{eq:explaination_2}
	t^\lambda \int_\Omega \ue^p(x,t) \d x \leq C(T,p,\lambda)  	
	\stext{ and }
	\int_0^t s^\lambda \int_\Omega \ue(x,s)^p |\grad \ve(x,s)|^2 \d x \d s \leq C(T,p,\lambda)
\end{equation}
In \Cref{section:solution_construction}, we then use the first half of (\ref{eq:explaination_2}) as the basis of a straightforward bootstrap argument employing semigroup theory as well as standard parabolic regularity theory to gain sufficiently strong parabolic Hölder space bounds for our approximate solutions. Using standard compact embedding properties of Hölder spaces, this then allows us to construct functions $u, v \in C^{2,1}(\overline{\Omega}\times(0, \infty))$ as the limit of our approximate solutions by diagonal sequence argument, which immediately inherit their classical solution properties from said approximate solutions.
\\[0.5em]
In \Cref{section:initial_data_continuity}, it now remains to show that the functions $u$ and $v$ also satisfy the continuity properties (\ref{u_continuity}) and (\ref{v_continuity}), respectively. Our approach here is twofold. We first argue that our approximate solutions were already uniformly continuous in $t = 0$ in appropriate topologies and then only need to show that this uniform continuity property survives the limit process. While this is fairly straightforward for the family $(\ve)_{\eps \in (0,1)}$, treating the family $(\ue)_{\eps \in (0,1)}$ is a bit more involved due to the problematic taxis term in the first equation of (\ref{problem}). To overcome this problem and control the taxis term, we make decisive use of the second part of (\ref{eq:explaination_2}) to show in \Cref{lemma:continuity_taxis_bound} that
$
	\int_0^t \|\ue(\cdot, s) \grad \ve(\cdot, s)\|_\L{1} \d s	
$
becomes uniformly small as $t\searrow 0$. 

\section{Regularized initial data and approximate solutions}
\label{section:approx_solutions}
For the remainder of this paper, we fix a bounded domain $\Omega \subseteq \R^n$, $n\geq 2$, with a smooth boundary, the system parameter $\chi > 0$ as well as the initial data $u_0 \in \Mp$ and $v_0 \in L^\infty(\Omega)$ with $m \defs u_0(\overline{\Omega}) > 0$, $v_0 \geq 0$ and $v_0 \not\equiv 0$. Constants and parameters in the results and proofs of this paper will only implicitly depend on the above parameters. Any other dependencies will be made explicit.
\\[0.5em]
As already expanded upon in the introduction, we will construct our desired solutions as the limit of a family of approximate solutions obtained by using the already established local existence theory for (\ref{problem}) from \cite{TaoBoundednessChemotaxisModel2011}. To this end, we begin by fixing a family $(u_{0,\eps})_{\eps \in (0,1)} \in C^\infty(\overline{\Omega})$ such that all $u_{0, \eps}$ are positive, that  
\begin{equation}\label{u0_props}
    u_{0, \eps} \rightarrow  u_0 \;\;\;\; \text{ in } \Mp \text{ as } \eps \searrow 0 \stext{ and that }  \int_\Omega u_{0, \eps} = u_0(\overline{\Omega}) \sfed m \text{ for all } \eps \in (0,1).
\end{equation}
For more details on how such an approximation can be achieved, we refer the reader to e.g.\ \cite[Remark 2.2]{HeihoffExistenceGlobalSmooth2021}. We further set $v_{0, \eps} \defs e^{\eps \laplace}v_0 \in C^\infty(\overline{\Omega})$ for each $\eps \in (0,1)$, where $(e^{t\laplace})_{t \geq 0}$ is the Neumann heat semigroup. Due to the maximum principle, this immediately guarantees positivity of all $v_{0, \eps}$ and further that 
\begin{equation}
    \|v_{0, \eps}\|_\L{\infty} \leq \|v_0\|_\L{\infty}
\end{equation}
for all $\eps \in (0,1)$. Moreover, the continuity properties of the heat semigroup ensure that 
\begin{equation}\label{v0_continuity}
	v_{0, \eps}	\rightarrow v_0 \;\;\;\; \text{ in }L^p(\Omega) \text{ for all } p\in (1,\infty)
\end{equation}
as $\eps \searrow 0$.
\\[0.5em]
As already mentioned, this now allows us to construct local solutions to (\ref{problem}) with initial data $(u_{0,\eps}, v_{0,\eps})$ using \cite[Lemma 2.1]{TaoBoundednessChemotaxisModel2011}.
\begin{lemma}\label{lemma:approx_exist}
	For each $\eps \in (0,1)$, there exists a maximal existence time $\tmaxeps \in (0,\infty]$ and a nonnegative classical solution 
	\[
		(\ue, \ve) \in \left( C^0(\overline{\Omega}\times[0,\tmaxeps)) \cap C^{2,1}(\overline{\Omega}\times (0,\tmaxeps)) \right)^2
	\] to (\ref{problem}) with initial data $(u_{0,\eps}, v_{0,\eps})$ and the following blow-up criterion:
	\begin{equation}\label{eq:blowup_criterion}
		\text{If }	\tmaxeps < \infty, \text{ then } \limsup_{t\nearrow \tmaxeps} \|\ue(\cdot, t)\|_\L{\infty}  = \infty.
	\end{equation}
	Further,
	\begin{equation}\label{eq:mass_conservation}
		\int_\Omega \ue(\cdot, t) = \int_\Omega u_{0,\eps} = u_0(\overline{\Omega}) = m \stext{ and } \|\ve(\cdot, t)\|_\L{\infty} \leq \|v_{0, \eps}\|_\L{\infty} \leq \|v_0\|_\L{\infty}
	\end{equation}
	for all $t \in (0,\tmaxeps)$.
\end{lemma}
\begin{proof}
	This lemma is a direct consequence of \cite[Lemma 2.1]{TaoBoundednessChemotaxisModel2011}.
\end{proof}\noindent
From this point onward we fix the approximate solutions $(\ue, \ve)$ established in the previous lemma as well as their maximal existence time $\tmaxeps$ for all $\eps \in (0,1)$.

\section{A priori estimates up to $t = 0$}
In this section, we will prove some key uniform a priori estimates for the approximate solutions up to $t=0$. As the arguments to achieve such estimates will in part be built on the techniques developed in \cite{TaoBoundednessChemotaxisModel2011} to construct classical solutions to (\ref{problem}) with smooth initial data, let us briefly reiterate the core ideas of said reference to give some context. The key innovation of \cite{TaoBoundednessChemotaxisModel2011} is to investigate the time evolution of the functional $\int_\Omega \ue^p \phi(\ve)$ with $\phi(s) \defs e^{(\beta s)^2}$, $s \geq 0$, $\beta > 0$, and $p = n+1$ by looking at its time derivative. By choosing $\beta$ in such a way that all terms in the resulting differential equation, which could potentially suggest growth of the functional, can be absorbed by a term of the form $\int_\Omega u^p \phi''(\ve) |\grad \ve|^2$ with negative sign due to some favorable pointwise estimates, it can be shown that the functional stays globally bounded. The resulting $L^{n+1}(\Omega)$ a priori information for $\ue$ then proves sufficient to rule out finite-time blowup by standard methods. Notably, the mentioned pointwise estimates only work under a smallness condition on the $L^\infty(\Omega)$ norm of the initial data $v_0$. 
\\[0.5em]
In our case, we evidently cannot expect the functional to be uniformly bounded even at $t = 0$. Nonetheless, a differential inequality for $\int_\Omega \ue^p \phi(\ve)$ of the same type as the one used in \cite{TaoBoundednessChemotaxisModel2011} will still prove invaluable for our approach. 
In contrast to the mentioned reference, we will only need said functional inequality for significantly smaller values of $p$ than $n+1$ and naturally want to make sure that the derivation is independent of $\eps$, which was not a concern in the original paper. By changing the choice of the parameter $\beta$, we will further improve the initial data condition for the second solution component compared to \cite{TaoBoundednessChemotaxisModel2011} as well. 
\begin{lemma}\label{lemma:magic_functional}
	If $v_0$ satisfies (\ref{v_0_condition}), then there exists $\delta \in (0,1)$ such that the inequality
	\begin{equation}\label{eq:energy}
		\frac{1}{p}\frac{\d}{\d t} \int_\Omega \ue^p \phi(\ve) + \frac{2\delta(p-1)}{p^2} \int_\Omega \phi(\ve) |\grad \ue^\frac{p}{2}|^2 + \frac{\delta}{p}\int_\Omega \ue^p \phi''(\ve) |\grad \ve|^2 \leq 0
	\end{equation}
	holds for all $p \in (1, \frac{n}{2} + \delta]$, $t \in (0,\tmaxeps)$ and $\eps \in (0,1)$, where
	\begin{equation}\label{eq:phi_definition}
		\phi(s) \defs e^{(\beta s)^2} \;\;\;\; \text{ for all } s \geq 0 \stext{ with } \beta \equiv \beta(p) \defs \sqrt{\frac{\chi(p-1)}{4\|v_0\|_\L{\infty}}} .
	\end{equation}
\end{lemma}
\begin{proof}
	According to our assumption that $v_0$ satisfies (\ref{v_0_condition}), we can fix $\delta \in (0,1)$ such that
	\begin{equation}\label{eq:delta_definition}
		\|v_0\|_\L{\infty} \leq \frac{2}{(2+ \frac{1}{1-\delta})(n+2\delta)\chi}\left( 1 - \delta \right).
	\end{equation}
	Applying the differential equations from (\ref{problem}) as well as partial integration to the terms resulting from time differentiation of the functional $\int_\Omega \ue^p \phi(\ve)$, we then gain that
	\begin{align*}
		\frac{1}{p}\frac{\d}{\d t} \int_\Omega \ue^p \phi(\ve) &= \int_\Omega \ue^{p-1} \phi(\ve) \uet + \frac{1}{p}\int_\Omega \ue^p \phi'(\ve) \vet \\
		&= \int_\Omega \ue^{p-1} \phi(\ve) \laplace \ue - \chi \int_\Omega \ue^{p-1} \phi(\ve) \div (\ue \grad \ve)\\
		&\hphantom{=}+\frac{1}{p} \int_\Omega \ue^p \phi'(\ve) \laplace \ve - \frac{1}{p} \int_\Omega \ue^{p+1} \phi'(\ve)\ve \\
		&= - (p-1)\int_\Omega \ue^{p-2} \phi(\ve) |\grad \ue|^2 - \int_\Omega \ue^{p-1} \phi'(\ve) \grad \ve \cdot \grad \ue \\
		&\hphantom{=}+ \chi(p-1) \int_\Omega \ue^{p-1} \phi(\ve) \grad \ve \cdot \grad \ue + \chi \int_\Omega \ue^{p} \phi'(\ve)|\grad \ve|^2 \\
		&\hphantom{=}- \int_\Omega \ue^{p-1}\phi'(\ve) \grad \ue \cdot \grad \ve - \frac{1}{p} \int_\Omega \ue^{p} \phi''(\ve) |\grad \ve|^2 - \frac{1}{p} \int_\Omega \ue^{p+1} \phi'(\ve)\ve 
	\end{align*}
	for all $t\in(0,\tmaxeps)$ and $\eps \in (0,1)$. Given that the nonnegativity of $\phi'(s) = 2\beta^2 s \phi(s)$ for all $s\geq 0$ implies that 
	\[
		-\frac{1}{p} \int_\Omega \ue^{p+1} \phi'(\ve)\ve  \leq 0,
	\]
	it further follows that 
	\begin{align*}
		&\frac{1}{p}\frac{\d}{\d t} \int_\Omega \ue^p \phi(\ve) + (p-1)\int_\Omega \ue^{p-2} \phi(\ve) |\grad \ue|^2 + \frac{1}{p} \int_\Omega \ue^{p} \phi''(\ve) |\grad \ve|^2 \\
		&\leq -2 \int_\Omega \ue^{p-1}\phi'(\ve) \grad \ue \cdot \grad \ve + \chi(p-1) \int_\Omega \ue^{p-1} \phi(\ve) \grad \ve \cdot \grad \ue + \chi \int_\Omega \ue^{p} \phi'(\ve)|\grad \ve|^2
	\end{align*}
	for all $t\in(0,\tmaxeps)$ and $\eps\in(0,1)$ after some slight rearrangement. Applying Young's inequality to the first and second term on the right-hand side of the above inequality, we see that
	\[
		-2 \int_\Omega \ue^{p-1}\phi'(\ve) \grad \ue \cdot \grad \ve \leq \frac{(p-1)(1-\delta)}{2}\int_\Omega \ue^{p-2}\phi(\ve) |\grad \ue|^2 + \frac{2}{(p-1)(1-\delta)} \int_\Omega \ue^p \frac{{\phi'}^2(\ve)}{\phi(\ve)} |\grad \ve|^2
	\]
	as well as
	\[
		\chi(p-1) \int_\Omega \ue^{p-1} \phi(\ve) \grad \ve \cdot \grad \ue \leq \frac{p-1}{2}\int_\Omega \ue^{p-2}\phi(\ve) |\grad \ue|^2 + \frac{\chi^2(p-1)}{2}\int_\Omega \ue^p \phi(\ve) |\grad \ve|^2
	\]
	and thus further that
	\begin{align*}
		&\frac{1}{p}\frac{\d}{\d t} \int_\Omega \ue^p \phi(\ve) + \frac{\delta(p-1)}{2}\int_\Omega \ue^{p-2} \phi(\ve) |\grad \ue|^2 + \frac{1}{p} \int_\Omega \ue^{p} \phi''(\ve) |\grad \ve|^2 \\
		&\leq \frac{2}{(p-1)(1-\delta)} \int_\Omega \ue^p \frac{{\phi'}^2(\ve)}{\phi(\ve)} |\grad \ve|^2 + \frac{\chi^2(p-1)}{2}\int_\Omega \ue^p \phi(\ve) |\grad \ve|^2 + \chi \int_\Omega \ue^{p} \phi'(\ve)|\grad \ve|^2\\
		&= \int_\Omega \Big[\, \psi_1(\ve) + \psi_2(\ve) + \psi_3(\ve) \, \Big] \ue^p |\grad \ve|^2,
		\numberthis \label{eq:func_inequal}
	\end{align*}
	where
	\[
		\psi_1 \defs \frac{2{{\phi'}^2}}{(p-1)(1-\delta)\phi}, \;\;\;\; \psi_2 \defs \frac{\chi^2(p-1)}{2} \phi \stext{ and } \psi_3 \defs \chi \phi'
	\] 
	for all $t\in(0,\tmaxeps)$ and $\eps \in (0,1)$. 
	\\[0.5em]
	We will now establish suitable pointwise estimates for $\psi_1$, $\psi_2$ and $\psi_3$ by $\phi''$ on $[0,\|v_0\|_\L{\infty}]$ so that the term on the right-hand side of (\ref{eq:func_inequal}) can be absorbed by the corresponding term on the left-hand side of the same inequality. More specifically using that 
	\begin{equation}\label{eq:phi_props}
		\phi'(s) = 2\beta^2 s \phi(s) \stext{ and } \phi''(s) = 2\beta^2 \phi(s) + 2\beta^2 s \phi'(s) = 2\beta^2 (1 + 2\beta^2 s^2) \phi(s) 
	\end{equation}
	for all $s \geq 0$, we compute that
	\begin{align*}
		\frac{\psi_1(s)}{\frac{1}{p} \phi''(s)} &= \frac{2p}{(p-1)(1-\delta)} \frac{{\phi'(s)}^2}{\phi(s) \phi''(s)} =  \frac{2p}{(p-1)(1-\delta)}\frac{(2\beta^2 s)^2}{2\beta^2(1+2\beta^2 s^2)} \\
		&\leq \frac{2p}{(p-1)(1-\delta)}\frac{(2\beta^2 s)^2}{2\beta^2} =\frac{4p}{(p-1)(1-\delta)} \beta^2 s^2 
		\leq \frac{4p}{(p-1)(1-\delta)} \beta^2 \|v_0\|_\L{\infty}^2 = \frac{\chi p}{1-\delta} \|v_0\|_\L{\infty}
	\end{align*}
	due to our choice of $\beta$ in (\ref{eq:phi_definition}) and 
	\[
		\frac{\psi_2(s)}{\frac{1}{p} \phi''(s)} = \frac{\chi^2(p-1)p}{2} \frac{\phi(s)}{\phi''(s)} = \frac{\chi^2(p-1)p}{4\beta^2 (1 + 2\beta^2 s^2)} \leq \frac{\chi^2(p-1)p}{4\beta^2} = \chi p \|v_0\|_\L{\infty}
	\]
	again due to our choice of $\beta$ in (\ref{eq:phi_definition}) as well as
	\[
		\frac{\psi_3(s)}{\frac{1}{p} \phi''(s)} = p \chi \frac{\phi'(s)}{\phi''(s)} = \chi p \frac{2\beta^2 s}{2\beta^2 (1 + 2\beta^2 s^2)} \leq \chi p s	\leq \chi p \|v_0\|_\L{\infty} 
	\]
	for all $s\in[0,\|v_0\|_\L{\infty}]$. Combining the above three estimates with the fact that $p \leq \frac{n}{2}+\delta$ as well as (\ref{eq:delta_definition}), we gain that 
	\begin{equation}\label{eq:pointwise_estimate}
		\frac{\psi_1(s) + \psi_2(s) + \psi_3(s)}{\frac{1}{p} \phi''(s)} \leq (2+\tfrac{1}{1-\delta})\chi p \|v_0\|_\L{\infty} \leq\frac{(2+\frac{1}{1-\delta})(n + 2\delta)\chi }{2} \|v_0\|_\L{\infty} \leq (1-\delta)
	\end{equation}
	for all $s\in[0,\|v_0\|_\L{\infty}]$.
	Applying (\ref{eq:pointwise_estimate}) to the functional inequality (\ref{eq:func_inequal}) then yields
	\[
		\frac{1}{p}\frac{\d}{\d t} \int_\Omega \ue^p \phi(\ve) + \frac{\delta(p-1)}{2}\int_\Omega \ue^{p-2} \phi(\ve) |\grad \ue|^2 + \frac{\delta}{p}\int_\Omega \ue^{p} \phi''(\ve) |\grad \ve|^2 \leq 0 
	\]
	for all $t\in(0,\tmaxeps)$ and $\eps\in(0,1)$ as $\|\ve(\cdot, t)\|_\L{\infty} \leq \|v_0\|_\L{\infty}$ for all $t\in (0,\tmaxeps)$ according to \Cref{lemma:approx_exist}. This of course directly implies our desired result as $\ue^{p-2}|\grad \ue|  = \frac{4}{p^2}|\grad \ue^\frac{p}{2}|^2$.
\end{proof}\noindent
Having now derived the above differential inequality, we can see that it not only provides a favorable monotonicity property for the functional under consideration but also yields a dissipative term of the form $\int_\Omega \phi(\ve)|\grad \ue^\frac{p}{2}|^2$, which is similar to terms often used to derive the immediate smoothing properties of e.g.\ the heat equation as $\phi(\ve)$ is uniformly bounded from below and above. In our case, this dissipative character of the above inequality will in fact lead to a similar immediate uniform boundedness property for the functional independent of its initial value (and only relying on a uniform mass bound). While a standard approach to make use of the dissipative term to derive our desired boundedness property would be to show that the functional is a subsolution to an ordinary differential equation with superlinear decay (cf.\ e.g.\ \cite[Section 3 and 4]{HeihoffExistenceGlobalSmooth2021}), we will take a slightly different approach adapted from \cite[Lemma 3.3]{HeihoffDoesStrongRepulsion2022} to also gain some time-weighted a priori information about the third term in (\ref{eq:energy}). This information will later be crucial to prove that our construction yields solutions that are still connected to the initial data in a reasonable fashion. 
\begin{lemma}\label{lemma:time_scaled_functional}
	If $v_0$ satisfies (\ref{v_0_condition}), then there exists $\delta > 0$ such that the following holds: 
	\\[0.5em]
	For each $T \in (0,\infty)$, $p \in (1,\frac{n}{2}+\delta]$ and $\lambda > \frac{n}{2}(p-1)$, there exists $C \equiv C(T, p, \lambda) > 0$ such that
	\begin{equation}\label{eq:ue_a_priori}
		t^\lambda \int_\Omega \ue^p(x,t) \d x \leq C  	
	\end{equation}
	and 
	\begin{equation}\label{eq:grad_ue_a_priori}
		\int_0^t s^\lambda \int_\Omega \ue^p(x,s) |\grad \ve(x,s)|^2 \d x \d s \leq C	 
	\end{equation}
	for all $t\in (0,T)\cap(0,\tmaxeps)$ and $\eps \in (0,1)$.
\end{lemma}
\begin{proof}
	We begin by fixing $\delta > 0$ as in \Cref{lemma:magic_functional}, $p\in (1,\frac{n}{2} + \delta]$ as well as $\lambda > \frac{n}{2}(p-1)$.
	\\[0.5em]
	Multiplying the functional $\int_\Omega \ue^p \phi(\ve)$ by $t^\lambda$, with $\phi$ as in \Cref{lemma:magic_functional}, and differentiating regarding the time variable yields 
	\begin{align*}
		&\frac{1}{p}\frac{\d}{\d t}\left[ t^\lambda\int_\Omega \ue^p \phi(\ve)\right] + \frac{2\delta(p-1)}{p^2} t^\lambda\int_\Omega \phi(\ve) |\grad \ue^\frac{p}{2}|^2 + \frac{\delta}{p}t^\lambda\int_\Omega \ue^p \phi''(\ve) |\grad \ve|^2 \leq \frac{\lambda}{p} t^{\lambda - 1} \int_\Omega \ue^p \phi(\ve) \numberthis \label{eq:time_scaled_functional_ineq}
	\end{align*}
	for all $t\in(0,\tmaxeps)$ and $\eps \in (0,1)$ due to (\ref{eq:energy}). Noting that 
	\[
		\phi(\ve(\cdot, t)) = e^{(\beta \ve(\cdot, t))^2} \leq e^{(\beta \|v_0\|_\L{\infty})^2} \sfed K_1(p) = K_1	
	\]
	for all $t\in(0,\tmaxeps)$ and $\eps \in (0,1)$ with $\beta = \beta(p)$ as in \Cref{lemma:magic_functional}, we can apply a version of the Gagliardo--Nirenberg inequality allowing for a conveniently expanded parameter range (cf.\ \cite[Lemma 2.3]{LiBoundednessChemotaxisHaptotaxis2016}) to the problem term on the right-hand side of the above inequality to gain $K_2 = K_2(p) > 0$ such that 
	\begin{align*}
		\int_\Omega \ue^p \phi(\ve) &\leq K_1 \|\ue^\frac{p}{2}\|^2_\L{2} \leq K_1K_2 \|\grad \ue^\frac{p}{2}\|^\frac{2(p-1)}{p - 1+\frac{2}{n}}_\L{2} \|\ue^\frac{p}{2}\|^\frac{\frac{4}{n}}{p - 1+\frac{2}{n}}_\L{\frac{2}{p}}  + K_1 K_2 \|\ue^\frac{p}{2}\|^2_\L{\frac{2}{p}} \\
		&\leq K_3 \left( \int_\Omega |\grad \ue^\frac{p}{2}|^2 \right)^\frac{p-1}{p-1+\frac{2}{n}} + K_3
	\end{align*}
	for all $t\in(0,\tmaxeps)$ and $\eps \in (0,1)$ with $K_3 = K_3(p) \defs K_1 K_2\max(m^{2p/(np-n+2)}, m^p)$ due to the mass conservation property (\ref{eq:mass_conservation}). Using Young's inequality and the fact that $\phi \geq 1$, the above then yields
	\begin{align*}
		\int_\Omega \ue^p \phi(\ve) &\leq
		\tfrac{2\delta(p-1)}{\lambda p}t \int_\Omega |\grad \ue^\frac{p}{2}|^2 + K_4 t^{-\frac{n}{2}(p-1)} + K_3 \\
		&\leq \tfrac{2\delta(p-1)}{\lambda p}t \int_\Omega \phi(\ve)|\grad \ue^\frac{p}{2}|^2 + K_4 t^{-\frac{n}{2}(p-1)} + K_3
	\end{align*}
	with $K_4 = K_4(p, \lambda) \defs (\frac{\lambda p}{2\delta(p-1)})^{\frac{n}{2}(p-1)} K_3^{\frac{n}{2}(p-1+\frac{2}{n})}$, which applied to (\ref{eq:time_scaled_functional_ineq}) results in 
	\[
		\frac{1}{p}\frac{\d}{\d t} \left[ t^\lambda\int_\Omega \ue^p \phi(\ve) \right] + \frac{\delta}{p}t^\lambda\int_\Omega \ue^p \phi''(\ve) |\grad \ve|^2 \leq K_5 (t^{\lambda - \frac{n}{2}(p-1) - 1} + t^{\lambda - 1})
	\]
	for all $t\in (0,\tmaxeps)$ and $\eps \in (0,1)$ with $K_5 = K_5(p, \lambda) \defs \frac{\lambda}{p}\max(K_3, K_4)$. As our choice of $\lambda$ ensures that  
	\[
		\lambda  - 1 \geq \lambda - \tfrac{n}{2}(p-1) - 1 > - 1,
	\]	
	time integration combined with the fact that $\phi \geq 1$ and $\phi'' \geq 2\beta^2$ (cf.\ (\ref{eq:phi_props})) then completes the proof with $C(T,p,\lambda) \defs \frac{pK_5}{\delta(\lambda -\frac{n}{2}(p-1))} \max(1,\frac{1}{2\beta^2})(T^{\lambda - \frac{n}{2}(p-1)} + T^\lambda)$.
\end{proof}

\section{Construction of a solution candidate}\label{section:solution_construction}
Using the uniform bounds for the family $(\ue)_{\eps\in(0,1)}$ in $L^{\frac{n}{2}+\delta}(\Omega)$ established in \Cref{lemma:time_scaled_functional} as a baseline, we now devote this section to deriving sufficient parabolic Hölder bounds for our approximate solutions to construct a solution candidate for our main theorem by employing suitable compact embedding properties of said Hölder spaces. As the type of bootstrap argument, which we will use to achieve this, is rather well-established in the literature (cf.\ \cite{BellomoMathematicalTheoryKeller2015}) and will closely follow the steps laid out in \cite{HeihoffDoesStrongRepulsion2022}, we will keep most of the proofs in this section fairly brief and only go into more detail when our approach needs to deviate from the standard arguments to e.g.\ account for our lack of initial data regularity.
\\[0.5em]
Our first step is to improve the bounds gained from (\ref{eq:ue_a_priori}) using semigroup methods in a similar fashion to \cite[Lemma 4.1, Lemma 4.2]{HeihoffDoesStrongRepulsion2022}.
\begin{lemma}\label{lemma:grad_ve_bound}
	If $v_0$ satisfies (\ref{v_0_condition}), then there exists $q > n$ such that the following holds: 
	\\[0.5em]
	For each $t_0, t_1 \in \R$ with $t_1 > t_0 > 0$, there exists $C \equiv C(t_0, t_1) > 0$ such that 
	\[
		\|\ve(\cdot, t)\|_{W^{1,q}(\Omega)} \leq C 	
	\]
	for all $t \in (t_0, t_1)\cap(0,\tmaxeps)$ and $\eps \in (0,1)$.
\end{lemma}
\begin{proof}
	We begin by fixing $t_0 > 0$ and $t_1 > t_0$. According to \Cref{lemma:time_scaled_functional}, we can then fix $q > n$ as well as $K_1 = K_1(t_0, t_1) > 0$ such that 
	\[
		\|\ue(\cdot, t)\|_\L{\frac{q}{2}} < K_1
	\] 
	for all $t\in(t_0/2, t_1)\cap(0,\tmaxeps)$ and $\eps \in (0,1)$. Using the variation-of-constants representation of $\ve$ as well as the smoothing properties of the Neumann heat semigroup (cf.\ \cite[Lemma 1.3]{WinklerAggregationVsGlobal2010}), we can then find $K_2 = K_2(t_0, t_1) > 0$ such that 
	\begin{align*}
		\|\grad \ve(\cdot, t)\|_\L{q} 
		&\leq \|\grad e^{(t-\frac{t_0}{2})\laplace}\ve(\cdot, \tfrac{t_0}{2})\|_\L{q} + \int_{t_0/2}^t \|\grad e^{(t-s)\laplace}\ue(\cdot, s) \ve(\cdot, s)\|_\L{q} \d s \\
		&\leq K_2(t-\tfrac{t_0}{2})^{-\frac{1}{2}}\|\ve(\cdot, \tfrac{t_0}{2})\|_\L{q} + K_2 \int_{t_0/2}^t (t-s)^{-\frac{1}{2}-\frac{n}{2q}} \|\ue(\cdot, s)\ve(\cdot, s)\|_\L{\frac{q}{2}} \d s \\
		&\leq  K_2(\tfrac{t_0}{2})^{-\frac{1}{2}}|\Omega|^\frac{1}{q}\|v_0\|_\L{\infty} + K_1 K_2 \|v_0\|_\L{\infty} (\tfrac{1}{2} - \tfrac{n}{2q})^{-1} (t_1 - \tfrac{t_0}{2})^{\frac{1}{2} - \frac{n}{2q}}
	\end{align*}
	for all $t \in (t_0, t_1)\cap(0,\tmaxeps)$ and $\eps \in (0,1)$ as $q > n$ ensures that $-\frac{1}{2}-\frac{n}{2q} > -1$. Combined with (\ref{eq:mass_conservation}), this completes the proof.
\end{proof}
\begin{lemma}\label{lemma:ue_linfty_bound}
	If $v_0$ satisfies (\ref{v_0_condition}), then, for all $t_0, t_1 \in \R$ with $t_1 > t_0 > 0$, there exists $C \equiv C(t_0, t_1) > 0$ such that 
	\[
		\|\ue(\cdot, t)\|_\L{\infty} \leq C	
	\]
	for all $t \in (t_0, t_1)\cap(0,\tmaxeps)$ and $\eps \in (0,1)$.
\end{lemma}
\begin{proof}
	We begin by fixing $t_0 > 0$ and $t_1 > t_0$. According to \Cref{lemma:time_scaled_functional} and \Cref{lemma:grad_ve_bound}, we can then fix $q > n$ and $K_1 = K_1(t_0, t_1) > 0$ such that 
	\[
		\|\ue(\cdot, t)\|_\L{\frac{q}{2}} \leq K_1 \stext{ and } \|\grad \ve(\cdot, t)\|_\L{q} \leq K_1	
	\] 
	for all $t\in[t_0/2, t_1)\cap(0,\tmaxeps)$ and $\eps \in (0,1)$. Let then $r \in (n, q)$ be such that $r > \frac{q}{2}$ and further 
	\[
		M_{\eps}(T) \defs \sup_{t\in(t_0/2, t_1)\cap (0,T)} \left\|(t-\tfrac{t_0}{2})^\frac{n}{q}\ue(\cdot, t) \right\|_\L{\infty} < \infty
	\]
	for all $T \in (t_0,\tmaxeps)$ and $\eps \in (0,1)$.
	\\[0.5em]	
	Now using the variation-of-constants representation of $\ue$, the smoothing properties of the Neumann heat semigroup (cf.\ \cite[Lemma 1.3]{WinklerAggregationVsGlobal2010}) as well as the Hölder inequality, we can find $K_2 = K_2(t_0, t_1) > 0$ such that 
	\begin{align*}
		&\left\| (t-\tfrac{t_0}{2})^\frac{n}{q}\ue(\cdot, t) \right\|_\L{\infty} \\
		&\leq \left\|(t-\tfrac{t_0}{2})^\frac{n}{q}e^{(t-\tfrac{t_0}{2})\laplace}\ue(\cdot, \tfrac{t_0}{2}) \right\|_\L{\infty} + \left\| (t-\tfrac{t_0}{2})^\frac{n}{q}\int_{t_0/2}^t e^{(t-s)\laplace} \div (\ue(\cdot, s) \grad \ve(\cdot, s)) \d s  \right\|_\L{\infty} \\
		&\leq K_2 \|\ue(\cdot, \tfrac{t_0}{2})\|_\L{\frac{q}{2}} + K_2(t-\tfrac{t_0}{2})^\frac{n}{q} \int_{t_0/2}^t (t-s)^{-\frac{1}{2} - \frac{n}{2r} }\|\ue(\cdot, s)\grad \ve(\cdot, s)\|_\L{r} \d s \\
		&\leq K_1 K_2 + K_1 K_2 (t-\tfrac{t_0}{2})^\frac{n}{q} \int_{t_0/2}^t (t-s)^{-\frac{1}{2} - \frac{n}{2r} } \|\ue(\cdot, s)\|_\L{\frac{qr}{q-r}} \d s \\
		&\leq K_1 K_2 + K_1^{\frac{q+r}{2r}} K_2 (t-\tfrac{t_0}{2})^\frac{n}{q} \int_{t_0/2}^t (t-s)^{-\frac{1}{2} - \frac{n}{2r} }  \|\ue(\cdot, s)\|^\frac{3r-q}{2r}_\L{\infty}\d s\\
		&\leq K_1 K_2 + K_1^{\frac{q+r}{2r}} K_2  (M_{\eps}(T))^\frac{3r-q}{2r} (t-\tfrac{t_0}{2})^\frac{n}{q} \int_{t_0/2}^t (t-s)^{-\frac{1}{2} - \frac{n}{2r} }  (s-\tfrac{t_0}{2})^{-\frac{n}{q}\frac{3r-q}{2r}} \d s \numberthis \label{eq:ue_linfty_estimatation}
	\end{align*}
	for all $t \in (t_0/2, t_1)\cap(0,T)$, $T \in (t_0,\tmaxeps)$ and $\eps \in (0,1)$. Note here that $\frac{qr}{q-r} > \frac{q}{2}$ and $\frac{3r-q}{2r} \in (0,1)$ as $q > r > \frac{q}{2}$.
	\\[0.5em]
	Now applying the linear substitution $s \mapsto (t - \frac{t_0}{2})s + \frac{t_0}{2}$ to the last remaining integral term yields
	\[
		\int_{t_0/2}^t (t-s)^{-\frac{1}{2} - \frac{n}{2r} }  (s-\tfrac{t_0}{2})^{-\frac{n}{q}\frac{3r-q}{2r}} \d s = (t-\tfrac{t_0}{2})^{1 - \frac{1}{2}-\frac{n}{2r}-\frac{n}{q}\frac{3r-q}{2r}} \int_0^1 (1-s)^{-\frac{1}{2} - \frac{n}{2r} } s^{-\frac{n}{q}\frac{3r-q}{2r}} \d s = K_3 (t-\tfrac{t_0}{2})^{\frac{1}{2}-\frac{3n}{2q}}
	\] 
	for all $t\in(t_0/2, t_1)$ with $K_3 \defs B(1 -\frac{n}{q}\frac{3r-q}{2r}, \frac{1}{2} - \frac{n}{2r})$, where $B$ is the beta function. Notably this is only possible because $q > r > n$ ensures that $-\frac{1}{2} - \frac{n}{2r} > -1$ as well as $-\frac{n}{q}\frac{3r-q}{2r} > -\frac{3r-q}{2r} > -1$ and thus that all of the integrals involved are finite.
	Applying the above substitution to (\ref{eq:ue_linfty_estimatation}), we then gain that 
	\begin{equation}\label{eq:M_eps_ineq}
		M_{\eps}(T) \leq K_4 + K_4 (M_{\eps}(T))^\frac{3r-q}{2r} 
	\end{equation}
	for all $T \in (t_0,\tmaxeps)$ and $\eps \in (0,1)$ with $K_4 = K_4(t_0, t_1) \defs \max(K_1 K_2, K_1^{\frac{q+r}{2r}} K_2 K_3 (t_1 - \frac{t_0}{2})^{\frac{1}{2}-  \frac{n}{2q}})$ as $q > n$ implies that $\frac{1}{2} - \frac{3n}{2q} + \frac{n}{q} = \frac{1}{2}-  \frac{n}{2q} > 0$.
	As $\frac{3r-q}{2r} \in (0,1)$, the inequality in (\ref{eq:M_eps_ineq}) immediately implies that 
	\[
		M_{\eps}(T) \leq K_5
	\]
	for all $T \in (t_0,\tmaxeps)$ and $\eps \in (0,1)$ with $K_5 = K_5(t_0, t_1) \defs 2K_4 + (2K_4)^\frac{2r}{q-r}$ by Young's inequality.
	From this, it directly follows that
	\[
		\|\ue(\cdot, t)\|_\L{\infty} \leq (\tfrac{t_0}{2})^{-\frac{n}{q}}K_5
	\]
	for all $t \in (t_0, t_1)\cap(0,\tmaxeps)$ and $\eps \in (0,1)$.
\end{proof}\noindent
The above lemma is now sufficient to rule out finite-time blowup for our approximate solutions according to the blowup criterion in (\ref{eq:blowup_criterion}).
\begin{corollary}
	If $v_0$ satisfies (\ref{v_0_condition}), then $\tmaxeps = \infty$ for all $\eps \in (0,1)$.
\end{corollary}
\noindent
The remainder of the bootstrap argument is now a straightforward combination of standard parabolic regularity theory from the literature and we will thus keep the following proof relatively brief as it follows roughly the same outline as e.g.\ \cite[Lemma 4.3, Lemma 4.4]{HeihoffDoesStrongRepulsion2022}.
\begin{lemma}\label{lemma:boostrap}
	If $v_0$ satisfies (\ref{v_0_condition}), then, for each $t_0, t_1 \in \R$ with $t_1 > t_0 > 0$, there exist $C \equiv C(t_0, t_1) > 0$ and $\theta \equiv \theta(t_0, t_1) \in (0,1)$ such that 
	\[
		\|\ue\|_{C^{2+\theta,1+\frac{\theta}{2}}(\overline{\Omega}\times[t_0, t_1]) } \leq C	 \stext{ and } \|\ve\|_{C^{2+\theta,1+\frac{\theta}{2}}(\overline{\Omega}\times[t_0, t_1]) } \leq C	
	\]
	for all $\eps \in (0,1)$.
\end{lemma}
\begin{proof}
	Using the bounds established in \Cref{lemma:grad_ve_bound} and \Cref{lemma:ue_linfty_bound} allows us to 
	to apply parabolic Hölder regularity theory due to Porzio and Vespri (cf.\ \cite[Theorem 1.3]{PorzioVespriHoelder}) to find $\theta_1 = \theta_1(t_0, t_1) \in (0,1)$ such that both $(\ue)_{\eps\in(0,1)}$ and $(\ve)_{\eps\in(0,1)}$ are uniformly bounded in $C^{\theta_1, \frac{\theta_1}{2}}(\overline{\Omega}\times[\frac{t_0}{8},t_1])$, where the bounds here and similar bounds later in the proof only depend on $t_0$ and $t_1$.
	Using further parabolic Hölder regularity theory from \cite[p.170 and p.320]{LadyzenskajaLinearQuasilinearEquations1988} combined with a straightforward cutoff function argument, we then gain $\theta_2 = \theta_2(t_0, t_1) \in (0,1)$ such that the family $(\ve)_{\eps \in (0,1)}$ is uniformly bounded in $C^{2+\theta_2, 1+\frac{\theta_2}{2}}(\overline{\Omega}\times[\frac{t_0}{4}, t_1])$.
	We can then use a theorem due to Lieberman (cf.\ \cite[Theorem 1.1]{LiebermanHolderContinuityGradient1987}) and another cutoff function argument to find $\theta_3 = \theta_3(t_0, t_1) \in (0,1)$ such that the family $(\grad \ue)_{\eps \in (0,1)}$ is uniformly bounded in $C^{\theta_3, \frac{\theta_3}{2}}(\overline{\Omega}\times[\frac{t_0}{2},t_1])$.
	This then yet again allows us to apply regularity results from \cite[p.170 and p.320]{LadyzenskajaLinearQuasilinearEquations1988} in a similar fashion as before to gain $\theta_4 = \theta_3(t_0, t_1) \in (0,1)$ such that the family $(\ue)_{\eps\in(0,1)}$ is uniformly bounded in $C^{2+\theta_4, 1+\frac{\theta_4}{2}}(\overline{\Omega}\times[t_0, t_1])$.
	\\[0.5em]
	Setting $\theta = \theta(t_0, t_1) \defs \min(\theta_2, \theta_4)$ then completes the proof.
\end{proof}\noindent
Using the above bounds combined with the Arzelà--Ascoli compact embedding theorem, we can now construct candidates for our desired solutions as limits of their approximate counterparts. Given the strength of the bounds outside of $t = 0$ in the above lemma, it is then easy to see that all solution properties apart from those concerned with the initial data immediately translate from our approximate solutions to their limits.
\begin{lemma}\label{lemma:construction}
	If $v_0$ satisfies (\ref{v_0_condition}), then there exist a null sequence $(\eps_j)_{j\in\N} \subseteq (0,1)$ as well as nonnegative functions $u, v \in C^{2,1}(\overline{\Omega}\times(0,\infty))$ such that 
	\begin{equation}\label{eq:solution_upper_bound}
		\|v(\cdot, t)\|_\L{\infty} \leq \|v_0\|_\L{\infty}	
	\end{equation}
	for all $t > 0$, such that 
	\begin{align}
		\ue &\rightarrow u \;\;\;\; \text{ in } C^{2,1}(\overline{\Omega}\times[t_0, t_1]), \label{eq:u_solution_convergence} \\
		\ve &\rightarrow v \;\;\;\; \text{ in } C^{2,1}(\overline{\Omega}\times[t_0, t_1])  \label{eq:v_solution_convergence} 
	\end{align}
	for all $t_0, t_1 \in \R$ with $t_1 > t_0 > 0$ as $\eps = \eps_j \searrow 0$, and such that $(u,v)$ is a classical solution to (\ref{problem}).
\end{lemma}
\begin{proof}	
	For every $k \in \N$, we can find $\theta_k \in (0,1)$ such that the families $(\ue)_{\eps \in (0,1)}$ and $(\ve)_{\eps \in (0,1)}$ are uniformly bounded in $C^{2+ \theta_k,1 + \frac{\theta_k}{2}}(\overline{\Omega}\times[\frac{1}{k}, k])$ as a consequence of \Cref{lemma:boostrap}. As the spaces $C^{2+ \theta_k,1 + \frac{\theta_k}{2}}(\overline{\Omega}\times[\frac{1}{k}, k])$ embed compactly into the spaces $C^{2, 1}(\overline{\Omega}\times[\frac{1}{k}, k])$ for all $k \in \N$ due to e.g.\ the Arzelà--Ascoli theorem, a straightforward diagonal sequence argument then allows us to construct functions $u,v \in C^{2,1}(\overline{\Omega}\times(0,\infty))$ with our desired convergence properties (\ref{eq:u_solution_convergence}) and (\ref{eq:v_solution_convergence}) by successively extracting convergent subsequence in $C^{2, 1}(\overline{\Omega}\times[\frac{1}{k}, k])$ for increasing values of $k\in\N$. As the resulting convergence properties ensure that all terms in the system (\ref{problem}) converge pointwise on $\Omega\times(0,\infty)$ and $\partial\Omega\times(0,\infty)$, respectively, it also follows immediately that $(u,v)$ is a classical solution to (\ref{problem}) as this was already the case for each $(\ue, \ve)$. Nonegativity of $u$ and $v$ as well as the upper bound (\ref{eq:solution_upper_bound}) are further properties inherited from the approximate solutions due to our strong convergence properties and the basic solution properties laid out in \Cref{lemma:approx_exist}. 
\end{proof}

\section{Continuity in $t = 0$}\label{section:initial_data_continuity}
Having now constructed our solution candidate $(u,v)$ in \Cref{lemma:construction}, it remains to show that said solution candidate is connected to the initial data $(u_0, v_0)$ in a sensible fashion, i.e.\ we want to show that $u$ and $v$ are continuous in $t = 0$ in some appropriate topology and with the correct values at $t = 0$. As we will see, this will prove rather more involved for the first solution component than the second solution component due to some additional effort necessary to handle the taxis term. Nonetheless in both cases, the approach is the same at a fundamental level. We first show that our approximate solutions were already uniformly continuous in $t = 0$ in some appropriate sense and then show that this property survives the limit process due to its uniformity.
\\[0.5em]
We begin by treating the first solution component $u$. As our first step toward deriving the aforementioned uniform continuity property, we will make crucial use of the bound (\ref{eq:grad_ue_a_priori}) from \Cref{lemma:time_scaled_functional} to show that a certain space-time integral connected to the taxis becomes uniformly small as the upper bound of the integration time interval approaches zero.
\begin{lemma}\label{lemma:continuity_taxis_bound}
If $v_0$ satisfies (\ref{v_0_condition}), then there exist $C > 0$ and $\alpha > 0$ such that 
\[
	\int_0^t \|\ue(\cdot, s) \grad \ve(\cdot, s)\|_\L{1} \d s \leq Ct^\alpha	
\]
for all $t\in(0,1)$ and $\eps \in (0,1)$. 
\end{lemma}
\begin{proof}
	Let $\delta \in (0,1)$ be as in \Cref{lemma:time_scaled_functional}.
	We then fix $\lambda \in (\frac{\delta}{2}, 1)$ and $\alpha \in (0,1)$ such that \[
		\lambda + \alpha < 1 - \alpha.
	\] 
	Using Young's inequality as well as our choice of $\alpha$, we now note that
	\begin{align*}
		\int_0^t \|\ue(\cdot, s) \grad \ve(\cdot, s)\|_\L{1} \d s 
		&\leq \int_0^t s^{\lambda + \alpha} \int_\Omega \ue^\frac{n+\delta}{n}(x, s)|\grad \ve(x, s)|^2 \d x \d s + \int_0^t s^{-\lambda - \alpha} \int_\Omega \ue^\frac{n-\delta}{n}(x, s) \d x \d s \\
		&\leq t^\alpha \int_0^t s^{\lambda} \int_\Omega \ue^\frac{n+\delta}{n}(x, s)|\grad \ve(x, s)|^2 \d x \d s + m^\frac{n-\delta}{n}|\Omega|^\frac{\delta}{n}\int_0^t s^{\alpha - 1} \d s \\
		&= t^\alpha \int_0^t s^{\lambda} \int_\Omega \ue^\frac{n+\delta}{n}(x, s)|\grad \ve(x, s)|^2 \d x \d s + \tfrac{1}{\alpha}m^\frac{n-\delta}{n}|\Omega|^\frac{\delta}{n} t^{\alpha} \label{eq:u_grad_v_ineq} \numberthis
	\end{align*}
	for all $t\in(0,1)$ and $\eps\in(0,1)$ due to the mass conservation property (\ref{eq:mass_conservation}). As $\lambda > \frac{\delta}{2} = \frac{n}{2}(\frac{n+\delta}{n} - 1)$ and $\frac{n+\delta}{n} = 1 + \frac{\delta}{n} \leq \frac{n}{2} + \delta$, we can then use \Cref{lemma:time_scaled_functional} to gain $K > 0$ such that 
	\[
		\int_0^t s^{\lambda} \int_\Omega \ue^\frac{n+\delta}{n}(x, s)|\grad \ve(x, s)|^2 \d x \d s \leq K
	\]
	for all $t\in(0,1)$ and $\eps\in(0,1)$, which combined with (\ref{eq:u_grad_v_ineq}) completes the proof.
\end{proof}\noindent
Using the fundamental theorem of calculus, the above property now almost immediately translates to uniform continuity of the first solution component of the approximate solutions in $t = 0$ in the vague topology. Due to the convergence properties already proven in \Cref{lemma:construction} as well as due to the properties of our initial data approximation, this then readily gives us our desired continuity property for $u$.
\begin{lemma}\label{lemma:u_continuity}
If $v_0$ satisfies (\ref{v_0_condition}) and $u$ is the function constructed in \Cref{lemma:construction}, then 
\[
	u(\cdot, t) \rightarrow u_0 \;\;\;\; \text{ in } \Mp	
\]
as $t \searrow 0$, where $\Mp$ is the set of positive Radon measures on $\overline{\Omega}$ with the vague topology.
\end{lemma}
\begin{proof} Let $\eta > 0$ and $\phi \in C^2(\overline{\Omega})$ with $\grad \phi\cdot\nu = 0$ on $\partial \Omega$ be fixed, but arbitrary.
\\[0.5em]
Using the fundamental theorem of calculus, the first equation in (\ref{problem}), partial integration as well as \Cref{lemma:continuity_taxis_bound}, we gain $K > 0$ and $\alpha > 0$ such that
\begin{align*}
	\left| \int_\Omega \ue(\cdot, t)\phi - \int_\Omega u_{0, \eps}\phi \right| 
	&= \left| \int_0^t \int_\Omega \uet \phi \right| 
	= \left| \int_0^t \int_\Omega \left[ \laplace \ue - \chi \div (\ue \grad \ve) \right]\phi \right| \\
	&= \left| \int_0^t \int_\Omega \ue \laplace \phi + \chi \int_0^t\int_\Omega \ue \grad \ve \cdot \grad \phi \right| \\
	&\leq m \|\laplace \phi\|_\L{\infty} t + \chi \|\grad \phi\|_\L{\infty} \int_0^t\int_\Omega \|\ue\grad \ve\|_\L{1} \\
	&\leq  m \|\laplace \phi\|_\L{\infty} t + \chi K \|\grad \phi\|_\L{\infty} t^\alpha
\end{align*}
for all $t\in(0,1)$ and $\eps \in (0,1)$. Thus, we can fix $t_0 \in (0,1)$ such that 
\[
	\left| \int_\Omega \ue(\cdot, t)\phi - \int_\Omega u_{0, \eps}\phi \right| \leq \frac{\eta}{3}
\]
for all $t\in(0,t_0)$ and $\eps \in (0,1)$. Using the convergence properties from (\ref{u0_props}) as well as \Cref{lemma:construction}, we can then, for each $t\in (0,t_0)$, fix $\eps(t) \in (0,1)$ such that 
\[
	\left| \int_\Omega u_{\eps(t),0}\phi - \int_{\overline{\Omega}} \phi \d u_0  \right| \leq \frac{\eta}{3} \stext{ and } \left| \int_\Omega u(\cdot, t) \phi - \int_\Omega u_{\eps(t)}(\cdot, t)\phi \right| \leq \frac{\eta}{3}.
\]
Combining the above estimates, we then conclude that 
\begin{align*}
	\left| \int_\Omega u(\cdot, t)\phi - \int_{\overline{\Omega}} \phi \d u_0 \right| 
	&\leq \left| \int_\Omega u(\cdot, t) \phi - \int_\Omega u_{\eps(t)}(\cdot, t)\phi \right|\\
	&+ \left| \int_\Omega u_{\eps(t)}(\cdot, t)\phi - \int_\Omega u_{\eps(t), 0}\phi \right| \\
	&+\left| \int_\Omega u_{\eps(t),0}\phi - \int_{\overline{\Omega}} \phi \d u_0  \right| \leq \frac{\eta}{3} + \frac{\eta}{3} +\frac{\eta}{3} = \eta
\end{align*}
for all $t\in(0,t_0)$. Given that the set of all functions $\phi \in C^2(\overline{\Omega})$ with $\grad \phi \cdot \nu = 0$ on $\partial \Omega$ is dense in $C^0(\overline{\Omega})$ (as can be easily seen by using a similar approach to the one used to approximate the initial data $v_0$ in \Cref{section:approx_solutions}), this immediately implies our desired result.
\end{proof}\noindent
To now prove a corresponding continuity property in $t = 0$ for the second solution component, we will use semigroup methods combined with the baseline bounds established in \Cref{lemma:approx_exist} to facilitate a similar argument to the one laid out above. Notably, the following reasoning is somewhat streamlined by our choice of initial data approximation for $v_0$, which is inherently compatible with the action of the Neumann heat semigroup.
\begin{lemma}\label{lemma:v_continuity}
	If $v_0$ satisfies (\ref{v_0_condition}) and $v$ is the function constructed in \Cref{lemma:construction}, then 
	\[
		v(\cdot, t) \rightarrow v_0 \;\;\;\; \text{ in } L^p(\Omega) \text{ for all } p\in[1,\infty)
	\]
	as $t \searrow 0$.
\end{lemma}
\begin{proof}
	Let $\eta > 0$ be fixed, but arbitrary.
	\\[0.5em]
	Using the variation-of-constants representation of the second equation in (\ref{problem}) combined with the baseline bounds established in \Cref{lemma:approx_exist} as well as the fact that $v_{0,\eps} = e^{\eps \laplace}v_0$ by definition, we can now estimate as follows:
	\begin{align*}
		\|v_{0, \eps} - \ve(\cdot, t)\|_\L{1} &\leq \| v_{0, \eps} - e^{t\laplace} v_{0, \eps}  \|_\L{1} + \int_0^t\left\| e^{(t-s)\laplace} \ue(\cdot, s) \ve(\cdot, s) \right\|_\L{1} \d s \\
		&\leq \|e^{\eps\laplace}(v_0 - e^{t\laplace} v_0)\|_\L{1} +  \int_0^t \| \ue(\cdot, s) \ve(\cdot, s) \|_\L{1} \d s \\
		&\leq \|v_0 - e^{t\laplace} v_0\|_\L{1} + m \|v_0\|_\L{\infty} t
	\end{align*}
	for all $t > 0$ and $\eps \in (0,1)$. Due to the continuity properties of the Neumann heat semigroup in $t = 0$, the above allows us to find $t_0 > 0$ such that 
	\[
		\|v_{0, \eps} - \ve(\cdot, t)\|_\L{1} \leq 2^{1-p}\|v_0\|^{1-p}_\L{\infty} \frac{\eta^p}{3^p} 	
	\]
	for all $\eps \in (0,1)$ and $t \in (0,t_0)$. By application of the Hölder inequality, we then further gain that 
	\begin{align*}
		\|v_{0, \eps} - \ve(\cdot, t)\|_\L{p} 
		&\leq \|v_{0, \eps} - \ve(\cdot, t)\|^\frac{1}{p}_\L{1} \|v_{0, \eps} - \ve(\cdot, t)\|^\frac{p-1}{p}_\L{\infty} \\
		&\leq  2^\frac{p-1}{p} \|v_{0, \eps} - \ve(\cdot, t)\|^\frac{1}{p}_\L{1} \|v_0\|^\frac{p-1}{p}_\L{\infty} \\
		&\leq \frac{\eta}{3} \numberthis \label{eq:v_eta_third_1}
	\end{align*}
	for all $\eps \in (0,1)$ and $t \in (0,t_0)$.
	\\[0.5em]
	Using the convergence properties laid out in (\ref{v0_continuity}) as well as \Cref{lemma:construction}, we can, for each $t \in (0,t_0)$, find $\eps(t) \in (0,1)$ such that 
	\[
		\|v_0 - v_{0, \eps(t)}\|_\L{p} \leq \frac{\eta}{3} \stext{ and } \|v(\cdot, t) - v_{\eps(t)}(\cdot, t)\|_\L{p} \leq \frac{\eta}{3}.
	\]
	Combined with (\ref{eq:v_eta_third_1}), this gives us that 
	\begin{align*}
		\|v_0 - v(\cdot, t)\|_\L{p} 
		&\leq \|v_0 - v_{0, \eps(t)}\|_\L{p} + \|v_{0, \eps(t)} - v_{\eps(t)}(\cdot, t)\|_\L{p} + \| v_{\eps(t)}(\cdot, t) - v(\cdot, t)\|_\L{p} \\
		&\leq \frac{\eta}{3} + \frac{\eta}{3} + \frac{\eta}{3} = \eta
	\end{align*}
	for all $t\in (0,t_0)$. As $\eta > 0$ was arbitrary, this completes the proof.
\end{proof}
\noindent
Having at this point proven all parts of \Cref{theorem:main} individually, its proof can now be presented in a rather swift fashion.
\begin{proof}[Proof of \Cref{theorem:main}]
	Assume that $v_0$ satisfies (\ref{v_0_condition}). Then let $u,v \in C^{2,1}(\overline{\Omega}\times(0,\infty))$ be the functions constructed in \Cref{lemma:construction} under this assumption. According to the same lemma, $(u,v)$ is already a solution to (\ref{problem}). \Cref{lemma:u_continuity} and \Cref{lemma:v_continuity} then further ensure that the continuity properties (\ref{u_continuity}) and (\ref{v_continuity}) hold for $u$ and $v$ as well. With this, all necessary properties for $u$ and $v$ are proven.
\end{proof}

\section*{Acknowledgment} The author acknowledges support of the \emph{Deutsche Forschungsgemeinschaft} in the context of the project \emph{Fine structures in interpolation inequalities and application to parabolic problems}, project number 462888149.

\end{document}